\documentclass{amsart}
\usepackage{amsfonts}
\usepackage{bbm}
\usepackage{mathrsfs}
\usepackage{color}
\usepackage{amsmath}
\usepackage{bigdelim}
\usepackage{multirow}
\usepackage{amssymb}
\usepackage{indentfirst}
\usepackage{CJK}
\usepackage{fancyhdr}
\usepackage{amscd,amssymb,amsmath,graphicx,verbatim,xy}
\usepackage[TS1,OT1,T1]{fontenc}

\newtheorem{theorem}{Theorem}[section]
\newtheorem{lemma}[theorem]{Lemma}

\theoremstyle{definition}
\newtheorem{definition}[theorem]{Definition}

\theoremstyle{remark}
\newtheorem{remark}[theorem]{Remark}
\numberwithin{equation}{section}

\begin{document}

\title[The right invariant metric on the analytic automorphism group]{The right invariant metric on the analytic automorphism group of the unit open disk induced by maximal modulus }

\author{Yue Xin}
\address{Yue Xin, School of Mathematical Science, Heilongjiang University, 150080, Harbin, P. R. China}
\email{2024104@hlju.edu.cn}

\author{Yan Li}
\address{Yan Li, Sixth High School of CAIDA, 130011, Changchun, P. R. China}
\email{41905193@qq.com}

\author{Bingzhe Hou}
\address{Bingzhe Hou, School of Mathematics, Jilin University, 130012, Changchun, P. R. China}
\email{houbz@jlu.edu.cn}
	
\date{}
\thanks{}
\begin{abstract}
In this paper, we study the right invariant metric $d_{H^{\infty}}$ on the analytic automorphism group $\rm{Aut}(\mathbb{D})$ of the unit open disk $\mathbb{D}$ induced by maximal modulus, that is,
$d_{H^{\infty}}(\varphi, \psi)=\sup_{z\in\mathbb{D}}|\varphi(z)-\psi(z)|$ for any $\varphi, \psi\in \rm{Aut}(\mathbb{D})$.
We give the explicit formula of the right invariant metric $d_{H^{\infty}}$ and characterize the almost regular Finsler geometric structure of $(\rm{Aut}(\mathbb{D}), d_{H^{\infty}})$.
\end{abstract}
\maketitle

\section{Introduction}

The research into invariant Riemannian metrics on Lie groups began with the classical work of Milnor \cite{Mil}, which investigated the geometric characterization of various types of Lie groups. From then on, the invariant metric on a group has been an important tool for studying group actions, representation theory, and geometric structures. Notice that there are many meaningful invariant Finsler (non-Riemannian) metrics on groups. As pointed out by S.-S. Chern \cite{Chern}, Finsler geometry is just Riemannian geometry without the quadratic restriction. The study of Finsler spaces has many applications in Physics and Biology (see \cite{Ant} for example). In particular, G. Randers \cite{Ra} introduced a very interesting type of Finsler structures from general relativity in 1941, named Randers metric now.  Then, the Randers metric was extensively studied in both theory and application (for example, see \cite{CS}, \cite{Cheng}, \cite{Li} and the references therein). So far,  many invariant Finsler (non-Riemannian) metrics on groups have been introduced and studied. For instance, L. Huang \cite{Huang} studied Ricci curvatures of left invariant Finsler metrics on Lie groups, S. Vukmirovi\'{c} \cite{Vu15} and M. Nasehi \cite{Nas} studied the invariant metrics of the Heisenberg groups and their Finsler geometry. The Finsler geometry of invariant metrics are also widely studied on homogeneous manifolds (\cite{DH1}, \cite{DH2}, \cite{DH3} and \cite{Tan}) and Teichm\"{u}ller space (see \cite{Pap}, \cite{SZ} and the references therein). In addition, a class of weakened Finsler metrics in the setting of almost regularity was also studied, such as \cite{Zhou}, \cite{CL} and \cite{TN}, where the almost regularity means the Finsler metric is $C^{\infty}$ except for some directions.

We are interested in the analytic automorphism group $\rm{Aut}(\mathbb{D})$ of the unit open disk $\mathbb{D}$. As well known, $\rm{Aut}(\mathbb{D})$ is isomorphic to the analytic automorphism group of the hyperbolic plane, which is isomorphic to the projective special linear group $\rm{PSL}(2, \mathbb{R})$. The group $\rm{Aut}(\mathbb{D})$ describes the symmetry of hyperbolic geometry on the Poincar\'{e} disk, and then it plays an important role in symmetry analysis of quantum mechanics (conformal field theory), transformation theory in computer graphics, symmetry structure analysis in circuit design and so on. In addition, the analytic automorphism group $\rm{Aut}(\mathbb{D})$ is naturally closely related to the study of analytic functions on $\mathbb{D}$ in complex analysis and functional analysis (see \cite{Gar} for example).

Notice that each $\varphi\in\rm{Aut}(\mathbb{D})$ could be written as the following form
$$
\varphi(z)=\mathrm{e}^{\mathbf{i}\xi}\frac{z-u}{1-\overline{u}z}, \ \ \ \text{for some} \ \xi\in\mathbb{R} \ \text{and} \ u\in\mathbb{D}.
$$
The correspondence $\varphi\mapsto (\xi, u)$ naturally induces a local coordinate system of the manifold $\rm{Aut}(\mathbb{D})$.
For convenience, we denote the analytic automorphism $\mathrm{e}^{\mathbf{i}\xi}\frac{z-u}{1-\overline{u}z}$ by $f_{\xi, u}(z)$,
and for any $u\in\mathbb{D}$, denote
$$
\varphi_{u}(z)=f_{0,u}(z)=\frac{z-u}{1-\bar{u}z}.
$$
In addition, recall that the quasi-hyperbolic metric $\rho$ on $\mathbb{D}$ is defined by
\[
\rho(u, v)=\left|\frac{u-v}{1-\bar{u}v}\right|, \ \ \ \text{for any} \  \ u, v\in\mathbb{D}.
\]

In this paper, we aim to introduce a metric $d_{H^{\infty}}$ on the analytic automorphism group $\rm{Aut}(\mathbb{D})$ from maximal modulus ($H^{\infty}$-norm), that is,
$$
d_{H^{\infty}}(\varphi, \psi)=\sup\limits_{z\in\mathbb{D}}|\varphi(z)-\psi(z)|, \ \ \  \ \text{for any} \ \varphi, \psi\in \rm{Aut}(\mathbb{D}).
$$
Let $G$ be a group and $d$ be a metric on $G$. If for any $g\in G$,
\[
d(g_1g, g_2g)=d(g_1, g_2), \ \ \ \ \text{for any} \ g_1,g_2\in G,
\]
then we call the metric $d$  a right invariant metric on $G$. From the perspective of function theory, it is not difficult to see that $d_{H^\infty}$ is a right invariant metric on $\rm{Aut}(\mathbb{D})$. We will give the explicit formula of the right invariant metric $d_{H^{\infty}}$ and show the Finsler geometric structure of $(\rm{Aut}(\mathbb{D}), d_{H^{\infty}})$. Our main result is as follows.

\vspace{2mm}

\noindent \textbf{Main Theorem.} \ \  Given any $\xi, \eta\in\mathbb{R}$ and any $u,v\in\mathbb{D}$. Let
\[
\lambda=\mathrm{e}^{\mathbf{i}(\eta-\xi)}\cdot\frac{1-\bar{u}v}{1-u\bar{v}}.
\]
Then, we have the explicit formula of the distance $d_{H^\infty}(f_{\xi, u},f_{\eta, v})$ on $\rm{Aut}(\mathbb{D})$ as follows,
\begin{enumerate}
  \item if $|\lambda+1|\leq 2\rho(u,v)$,
  \begin{equation*}
    d_{H^\infty}(f_{\xi, u},f_{\eta, v})=2;
  \end{equation*}
  \item if $|\lambda+1|>2\rho(u,v)$,
  \begin{equation*}\label{e*}
      d_{H^\infty}(f_{\xi, u},f_{\eta, v})=\left|(1-2\rho^2(u,v)-2\mathbf{i}\rho(u,v)\sqrt{1-\rho^2(u,v)})-(\textrm{Re}\lambda+\mathbf{i}|\textrm{Im}\lambda|)\right|.
  \end{equation*}
  \end{enumerate}
Furthermore, the induced geometric structure $F(f_{\xi,u},(s,h))$ is a non-Riemannian reversible almost regular Finsler structure, whose explicit formula is as follows,
\[
F(f_{\xi, u}, (s,h))=\frac{2|h|+|2\mathrm{Im}(\bar{u}h)-s(1-|u|^2)|}{1-|u|^2}.
\]
Moreover, the metric space $(\rm{Aut}(\mathbb{D}), d_{H^\infty})$ is not a geodetic space.
\vspace{2mm}

This article is organized as follows. In Section 2, we review some fundamental concepts and results in  Finsler geometry. In Section 3, we show the right invariance of the metric $d_{H^\infty}$ on the analytic automorphism group $\textrm{Aut}(\mathbb D)$ and give the proof of the Main Theorem.

\section{Preliminaries}

In this section, let us review some fundamental concepts of Finsler geometry (one could refer to \cite{chern}, \cite{DSZ} and \cite{shenshen} for more details).

\begin{definition}[Minkowski norm]
	Let $V$ be an $n$-dimensional real vector space. A function $F=F(y)$ on $V$ is called a Minkowski norm on $V$ if it satisfies the following conditions.
	\begin{enumerate}
		\item \ For any $y\in V$, $F(y\ge0)$ and $F(y)=0$ if and only if $y=0$.
		\item \ For any $y\in V$ and $t>0$, $F(t y)=t F(y)$.
		\item \ $F$ is $C^{\infty}$ on $V\backslash\{0\}$. Moreover, for any $y\in V\backslash\{0\}$, the bilinear symmetric form $g_y(\mathbf{u}, \mathbf{v}):=\frac{1}{2}\frac{\partial^2}{\partial s\partial t}[F^2(y+s\mathbf{u}+t\mathbf{v})]_{s=t=0}$ on $V$ is an inner product.
	\end{enumerate}
Furthermore, $(V,F)$ is called the Minkowski space and $g_y$ is called the fundamental form with respect to $y$.
\end{definition}

Let $M$ be an $n$-dimensional manifold. For any $x\in M$, denote by $T_x M$ the tangent space of $M$ at $x$, $TM=\cup_{x\in M}T_xM$ the tangent bundle of $M$. The elements in $TM$ can be write as $(x,y)$, where $y\in T_x M$. Denote by $\widetilde{TM}=TM\setminus0$ the slit tangent bundle of $M$. A Finsler structure of $M$ may be viewed as a smoothly varying family of Minkowski norms.

\begin{definition}[Finsler structure and almost regular Finsler structure]
	A function $F = F(x, y)$ on $TM$ is called  a Finsler structure or Finsler metric on $M$ if it satisfies the following conditions.
	\begin{enumerate}	
        \item (Regularity)\ $F(x,y)$ is $C^{\infty}$ on $\widetilde{TM}$.
		\item (Positive homogeneity)\ $F(x,t y)=t F(x,y)$, $t\in\mathbb{R}^+$, $(x,y)\in \widetilde{TM}$.
		\item (Strong convexity)\ The fundamental tensor $g_{ij}(x, y) = \frac{1}{2} \frac{\partial^2 F^2}{\partial y^i \partial y^j}$ is positive definite on $\widetilde{TM}$.
	\end{enumerate}
 A smooth manifold $M$ equipped with a Finsler structure $F$ is called a Finsler manifold, denoted by $(M,F)$. Moreover, if $F(x, y)=F(x,-y)$ for any $(x,y)\in \widetilde{TM}$, $F$ is called a reversible Finsler structure. More generally, if we replace $\widetilde{TM}$ by $TM\setminus E_0$, where $E_0$ is a proper sub-bundle of $TM$, we say that $F$ is an almost regular Finsler structure and $(M,F)$ is an almost regular Finsler manifold.
\end{definition}

\begin{definition}[Riemannian metric]
A Riemannian metric on an $n$-dimensional smooth manifold $M$ is a smooth symmetric positive definite covariant $2$-tensor field $g$ on $M$, satisfying the following properties.
	\begin{enumerate}
		\item (Bilinearity)\  For any point $p \in M$, and any tangent vectors $\mathbf{v}, \mathbf{w}, \mathbf{u} \in T_p M$ and scalars $a, b \in \mathbb{R}$, $g_p(a\mathbf{v} + b\mathbf{w}, \mathbf{u}) = a  g_p(\mathbf{v}, \mathbf{u}) + b  g_p(\mathbf{w}, \mathbf{u})$.
		\item (Symmetry)\ For any $p \in M$ and $\mathbf{v}, \mathbf{w} \in T_pM$, $g_p(\mathbf{v}, \mathbf{w}) = g_p(\mathbf{w}, \mathbf{v})$.
		\item (Positive definiteness)\ For any $p \in M$ and any non-zero tangent vector $\mathbf{v}\in T_pM$, $g_p(\mathbf{v}, \mathbf{v}) > 0$.
	\end{enumerate}
Furthermore, a Riemannian manifold is a smooth manifold equipped with a Riemannian metric.
\end{definition}

Let $(U; x^1, x^2, \dots, x^n)$ be a local coordinate system around a point $p \in M$, and let $\left\{ \frac{\partial}{\partial x^1}, \dots, \frac{\partial}{\partial x^n} \right\}$ be a basis for the tangent space $T_pM$. Then, the components of the Riemannian metric $g$ with respect to this basis are defined as $g_{ij}(p) = g_p\left( \frac{\partial}{\partial x^i}, \frac{\partial}{\partial x^j} \right)$. For any two tangent vectors $\mathbf{v} = v^i \frac{\partial}{\partial x^i}$ and $\mathbf{w} = w^j \frac{\partial}{\partial x^j}$, their inner product can be expressed as $g_p(\mathbf{v}, \mathbf{w}) = g_{ij}(p) v^i w^j$. Intuitively, $g$ assigns an inner product $g_p$ to each tangent space $T_pM$, and this inner product varies smoothly across $M$. Finsler structure $F$ on a Riemannian manifold satisfies $F(x,y)=\sqrt{g_{ij}(x)y^i y^j}$, then $\frac{1}{2}(F^2)_{y^i y^j}=g_{ij} (x)$, which is independent of tangent vectors. Then, a Riemannian manifold is a Finsler manifold with quadratic form constraints.

\begin{definition}
	Let $(M,F)$ be a Finsler manifold. Put
\[
			A_{ijk}:=\frac{F}{2}\frac{\partial g_{ij}}{\partial y^k}=\frac{1}{4}(F^2)_{y^iy^jy^k}(x,y).
 \]
The Cartan tensor $\mathcal{A}$ is
$
 \mathcal{A}=A_{ijk}d_{x^i}\otimes d_{x^j}\otimes d_{x^k}.
$
\end{definition}

\begin{definition}[Randers metric]
A Randers metric is a Finsler structure $F$ on $TM$ that has the form $F(x,y)=\alpha(x,y)+\beta(x,y)$,
in which $\alpha(x,y)=\sqrt{{a}_{ij}(x)y^iy^j}$ is a Riemannian metric and $\beta(x,y)={b}_{i}y^i$ is a $1$-form satisfying
\[
\|\beta\|_{\alpha}:=\sqrt{{a}^{ij}(x)b_ib_j}<1,
\]
where $(a^{ij})=(a_{ij})^{-1}$.
\end{definition}

In the above definition, the condition $\|\beta\|_{\alpha}<1$ guarantees the positive definiteness and strong convexity of $F$. Deicke's Theorem \cite{Dei} tells us that a Finsler structure $F$ is Riemannian if and only if $\mathcal{A}\equiv0$. Furthermore, a Randers metric $F$ is Riemannian if and only if $b_{i}=0$ for all $i$. It is easy to see that the previous statements also hold for almost regular Finsler structures and almost regular Randers metrics. It is worthy noticing that a reversible Randers metric must be a Riemannian metric, but a reversible almost regular Randers metric is possible to be a non-Riemannian metric.

\begin{definition}
	Let $(X, d)$ be a metric space, and let $L \in \mathbb{R}^+$. An $L$-geodesic in $X$ is an isometric embedding $\gamma: [0, L] \rightarrow X$, where the interval $[0, L]$ is equipped with the standard Euclidean metric $|t_2-t_1|$ on $\mathbb{R}$, satisfied that, for all $t_1, t_2 \in [0, L] $ with $t_1 \leq t_2$,  $d\left(\gamma(t_1), \gamma(t_2)\right) = |t_2 - t_1|$. A metric space $(X, d)$ is called a geodesic space if for any $x, x' \in X$, there exists a geodesic with starting point $x$ and ending point $x'$.
\end{definition}

\section{The geometric structure of $(\rm{Aut}(\mathbb{D}),d_{H^\infty})$}

In this section, we will show the right invariance of the metric $d_{H^\infty}$ on the analytic automorphism group $\textrm{Aut}(\mathbb D)$ and give the proof of the Main Theorem.

\begin{lemma}\label{ri}
$d_{H^\infty}$ is a right invariant metric on the analytic isomorphism group $\textrm{Aut}(\mathbb{D})$.
\end{lemma}

\begin{proof}
It is obvious that $d_{H^\infty}$ is the metric induced by $H^\infty$-norm on the analytic isomorphism group $\textrm{Aut}(\mathbb{D})$.
Given any $h\in\textrm{Aut}(\mathbb{D})$. Notice that $h$ is a bijection from $\mathbb{D}$ to itself. Then, for any $\varphi,\psi\in\textrm{Aut}(\mathbb{D})$,
\begin{equation*}
	d_{H^\infty}(\varphi\circ h, \psi\circ h)=\sup\limits_{z\in\mathbb{D}}\left|\varphi\circ h(z)-\psi\circ h(z)\right|=\sup\limits_{\omega\in\mathbb{D}}\left|\varphi(\omega)-\psi(\omega)\right|=d_{H^\infty}(\varphi, \psi).
\end{equation*}
This finishes the proof.
\end{proof}

\begin{proof}[\textbf{Proof of Main Theorem}]
\textbf{(1)} Firstly, let us give the explicit formula of the metric $d_{H^\infty}$ on $\rm{Aut}(\mathbb{D})$.
Given any $\xi, \eta\in\mathbb{R}$ and any $u,v\in\mathbb{D}$. Recall that
\[
\lambda=\mathrm{e}^{\mathbf{i}(\eta-\xi)}\cdot\frac{1-\bar{u}v}{1-u\bar{v}}.
\]
One can see that
\begin{align*}
	d_{H^\infty}(f_{\xi, u},f_{\eta, v})&=\sup\limits_{z\in\mathbb{D}}\left|\mathrm{e}^{{\bf i}\xi}\cdot\frac{z-u}{1-\bar{u}z}-\mathrm{e}^{{\bf i}\eta}\cdot\frac{z-v}{1-\bar{v}z}\right|\\
	&=\sup\limits_{z\in\mathbb{D}}\left|\frac{z-u}{1-\bar{u}z}-\mathrm{e}^{{\bf i}(\eta-\xi)}\cdot\frac{z-v}{1-\bar{v}z}\right|\\
	&=\sup\limits_{z\in\mathbb{D}}\left|z-\mathrm{e}^{{\bf i}(\eta-\xi)}\cdot\frac{\varphi_{-u}(z)-v}{1-\bar{v}\varphi_{-u}(z)}\right|\\
	&=\sup\limits_{z\in\mathbb{D}}\left|z-\mathrm{e}^{{\bf i}(\eta-\xi)}\cdot\frac{1-\bar{u}v}{1-u\bar{v}}\frac{z-\varphi_{u}(v)}{1-\overline{\varphi_{u}(v)}z}\right| \\
    &=\sup\limits_{\theta\in\mathbb{R}}\left|\mathrm{e}^{{\bf i}\theta}-\mathrm{e}^{{\bf i}(\eta-\xi)}\cdot\frac{1-\bar{u}v}{1-u\bar{v}}\cdot\frac{\mathrm{e}^{{\bf i}\theta}-\varphi_{u}(v)}{1-\overline{\varphi_{u}(v)}\mathrm{e}^{{\bf i}\theta}}\right| \\
    &=\sup\limits_{\theta\in\mathbb{R}}\left|1-\mathrm{e}^{{\bf i}(\eta-\xi)}\cdot\frac{1-\bar{u}v}{1-u\bar{v}}\cdot\frac{1-\varphi_{u}(v)\mathrm{e}^{-{\bf i}\theta}}{1-\overline{\varphi_{u}(v)}\mathrm{e}^{{\bf i}\theta}}\right| \\
    &=\sup\limits_{\theta\in\mathbb{R}}\left|1-\mathrm{e}^{{\bf i}(\eta-\xi)}\cdot\frac{1-\bar{u}v}{1-u\bar{v}}\cdot\frac{1-\rho(u,v)\mathrm{e}^{{\bf i}\theta}}{1-\rho(u,v)\mathrm{e}^{-{\bf i}\theta}}\right| \\
    &=\sup\limits_{\theta\in\mathbb{R}}\left|1-\lambda\cdot\frac{1-\rho(u,v)\mathrm{e}^{{\bf i}\theta}}{1-\rho(u,v)\mathrm{e}^{-{\bf i}\theta}}\right|.
 \end{align*}

As shown in Figure \ref{pic1}, we take two tangent lines from $1$ to the circle $\{z;~|z|=\rho(u,v)\}$, intersecting the unit circle at $\lambda_1$ and $\lambda_2$, respectively. Denote by $\delta$ the angle between these two tangent lines.

\begin{figure}[h]
	\centering
	\includegraphics[width=0.65\textwidth]{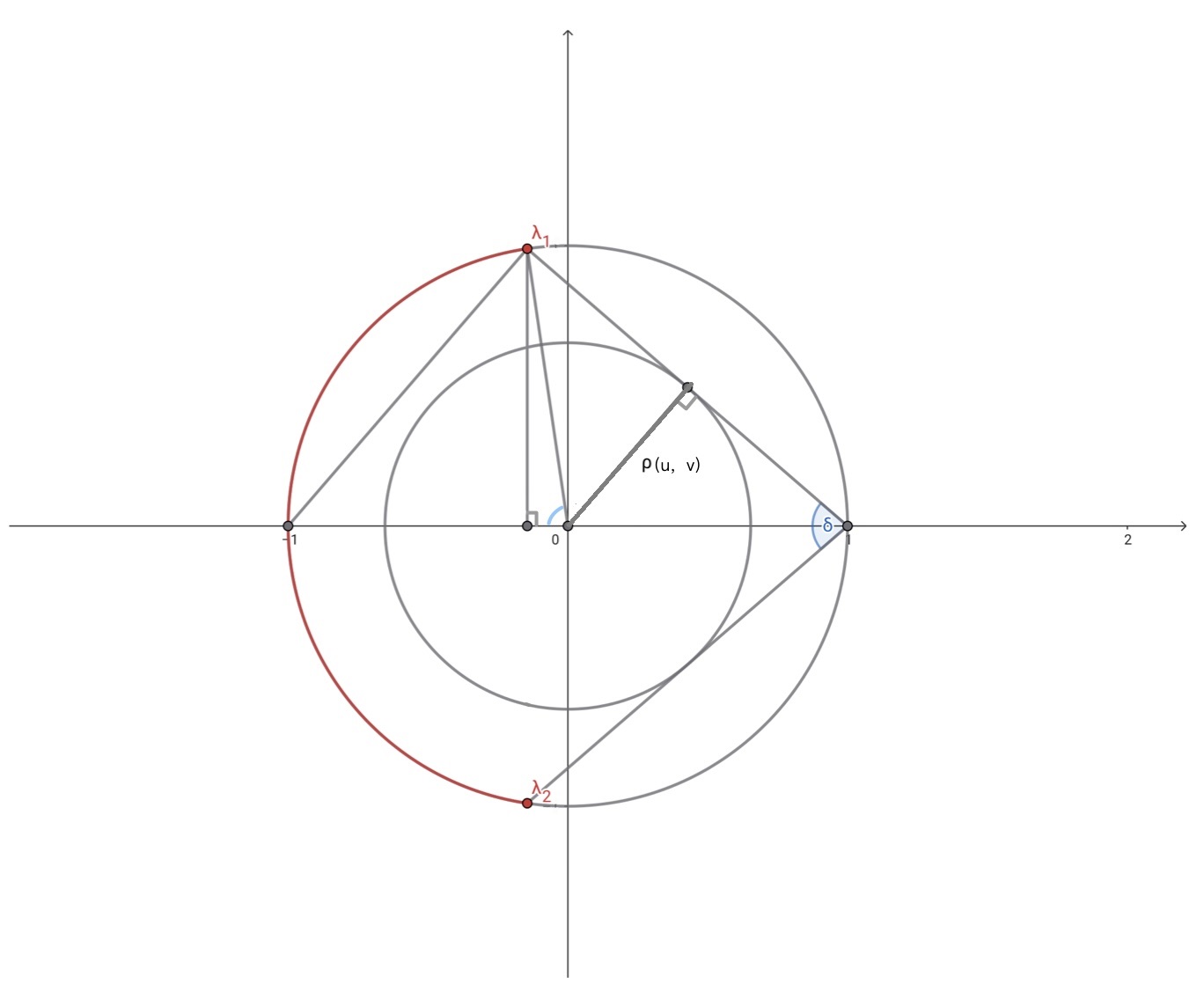}\\
	\caption{}\label{pic1}
\end{figure}

Then, one can see that
\[
\left\{\arg\left(\frac{1-\rho(u,v)\mathrm{e}^{{\bf i}\theta}}{1-\rho(u,v)\mathrm{e}^{-{\bf i}\theta}}\right);~\theta\in\mathbb{R}\right\}=[-\delta, \delta].
\]
Therefore, as shown in Figure \ref{pic1}, $d_{H^\infty}(f_{\xi, u},f_{\eta, v})=2$ if and only if $\lambda$ lies on the arc from $\lambda_1$ to $\lambda_2$ and passing through $-1$ (the red arc), i.e.,
\[
|\lambda+1|\leq 2\rho(u,v).
\]
Moreover, we have
\[
\textrm{Im}(\mathrm{e}^{{\bf i}\delta})=\rho(u,v)\cdot 2\sqrt{1-\rho^2(u,v)},
\]
and then
\[
\mathrm{e}^{{\bf i}\delta}=1-2\rho^2(u,v)+2\mathbf{i}\rho(u,v)\sqrt{1-\rho^2(u,v)}.
\]

Suppose that $|\lambda+1|>2\rho(u,v)$. Then,
\begin{align*}
	d_{H^\infty}(f_{\xi, u},f_{\eta, v})&=\sup\limits_{\theta\in\mathbb{R}}\left|1-\lambda\cdot\frac{1-\rho(u,v)\mathrm{e}^{{\bf i}\theta}}{1-\rho(u,v)\mathrm{e}^{-{\bf i}\theta}}\right| \\
    &=\left|1-(\textrm{Re}\lambda+\mathbf{i}|\textrm{Im}\lambda|)\mathrm{e}^{{\bf i}\delta} \right| \\
    &=\left|\mathrm{e}^{-{\bf i}\delta}-(\textrm{Re}\lambda+\mathbf{i}|\textrm{Im}\lambda|) \right| \\
    &=\left|(1-2\rho^2(u,v)-2\mathbf{i}\rho(u,v)\sqrt{1-\rho^2(u,v)})-(\textrm{Re}\lambda+\mathbf{i}|\textrm{Im}\lambda|) \right|.
 \end{align*}

\textbf{(2)} Let us give the explicit formula of the geometric structure on $\textrm{Aut}(\mathbb D)$ induced by the metric $d_{H^\infty}$.
Given $s\in\mathbb{R}$ and $h\in\mathbb{C}$ with $(s, h)$ being nontrivial. Let
\[
\lambda(\Delta t)=\mathrm{e}^{\mathbf{i}s\Delta t}\cdot\frac{1-\bar{u}(u+h\Delta t)}{1-u\overline{(u+h\Delta t)}}.
\]
Notice that
\begin{align*}
	&\lim\limits_{\Delta t\rightarrow 0^+}\frac{1-\lambda(\Delta t)}{\Delta t} \\
=&\lim\limits_{\Delta t\rightarrow 0^+}\frac{1-\mathrm{e}^{\mathbf{i}s\Delta t}\cdot\frac{1-\bar{u}(u+h\Delta t)}{1-u\overline{(u+h\Delta t)}}}{\Delta t} \\
=&\lim\limits_{\Delta t\rightarrow 0^+}\frac{1-\mathrm{e}^{\mathbf{i}s\Delta t}}{\Delta t}+\frac{\mathrm{e}^{\mathbf{i}s\Delta t}\left(1-\frac{1-\bar{u}(u+h\Delta t)}{1-u\overline{(u+h\Delta t)}}\right)}{\Delta t}\\
=&\lim\limits_{\Delta t\rightarrow 0^+}\frac{1-\mathrm{e}^{\mathbf{i}s\Delta t}}{\Delta t}+\lim\limits_{\Delta t\rightarrow 0^+}\mathrm{e}^{\mathbf{i}s\Delta t}\cdot\frac{\bar{u}h-u\bar{h}}{1-u\overline{(u+h\Delta t)}}\\
=&-\mathbf{i}s+\frac{2\mathbf{i}\mathrm{Im}(\bar{u}h)}{1-|u|^2}\\
=&\mathbf{i}\cdot\frac{2\mathrm{Im}(\bar{u}h)-s(1-|u|^2)}{1-|u|^2}.
 \end{align*}

If $2\mathrm{Im}(\bar{u}h)-s(1-|u|^2)\neq 0$, then for a sufficient small positive number $\Delta t$,
\[
\textrm{sgn}~(2\mathrm{Im}(\bar{u}h)-s(1-|u|^2))=\textrm{sgn}~\mathrm{Im}(1-\lambda(\Delta t))=-\textrm{sgn}~\mathrm{Im}\lambda(\Delta t).
\]
Then,
\begin{align*}
	&F(f_{\xi, u}, (s,h)) \\
=&\lim\limits_{\Delta t\rightarrow 0^+}\frac{d_{H^\infty}(f_{\xi, u},f_{\xi+s\Delta t, u+h\Delta t})}{\Delta t} \\
=&\lim\limits_{\Delta t\rightarrow 0^+}\frac{1}{\Delta t}\cdot\left|
\begin{aligned}
&(1-2\rho^2(u,u+h\Delta t)-2\mathbf{i}\rho(u,u+h\Delta t)\sqrt{1-\rho^2(u,u+h\Delta t)}) \\
&-(\textrm{Re}\lambda(\Delta t)+\mathbf{i}|\textrm{Im}\lambda(\Delta t)|)
\end{aligned} \right| \\
=&\lim\limits_{\Delta t\rightarrow 0^+}\frac{1}{\Delta t}\cdot\left|
\begin{aligned}
&1-2\rho^2(u,u+h\Delta t)-(\textrm{Re}\lambda(\Delta t)+\mathbf{i}\textrm{Im}\lambda(\Delta t)) \\
&-\textrm{sgn}~(\textrm{Im}\lambda(\Delta t))2\mathbf{i}\rho(u,u+h\Delta t)\sqrt{1-\rho^2(u,u+h\Delta t)}
\end{aligned} \right| \\
=&\lim\limits_{\Delta t\rightarrow 0^+}\frac{1}{\Delta t}\cdot\left|
\begin{aligned}
&(1-\lambda(\Delta t))-2\rho^2(u,u+h\Delta t) \\
&-\textrm{sgn}~(\textrm{Im}\lambda(\Delta t))2\mathbf{i}\rho(u,u+h\Delta t)\sqrt{1-\rho^2(u,u+h\Delta t)})
\end{aligned}\right| \\
=&\frac{2|h|+|2\mathrm{Im}(\bar{u}h)-s(1-|u|^2)|}{1-|u|^2}.
 \end{align*}

\textbf{(3)} Let us show that $F(f_{\xi, u}, (s,h))$ is a non-Riemannian reversible almost regular Finsler structure.
Let
 \begin{align*}
& E_0=\{(f_{\xi, u}, (s,h));~ 2\mathrm{Im}(\bar{u}h)-s(1-|u|^2)=0\}, \\
& E^+=\{(f_{\xi, u}, (s,h));~ 2\mathrm{Im}(\bar{u}h)-s(1-|u|^2)>0\}, \\
& E^-=\{(f_{\xi, u}, (s,h));~ 2\mathrm{Im}(\bar{u}h)-s(1-|u|^2)<0\}.
 \end{align*}
Then, $TM\setminus E_0=E^+\bigcup E^-$, $E_0$ is a proper sub-bundle of $TM$ and $F|_{TM\setminus E_0}$ is $C^\infty$.
Following from the explicit formula of $F(f_{\xi, u}, (s,h))$, it is obvious that $F(f_{\xi, u}, (s,h))$ satisfies the positive homogeneity and reversibility, i.e.,
\[
F(f_{\xi, u}, t(s,h))=tF(f_{\xi, u}, (s, h)), \ \ \  \text{for any} \ t\in\mathbb{R}.
\]

We use $G^{+}$ and $G^{-}$ to denote the fundamental tensors of $F$ on $E^+$ and $E^-$, respectively.
Let $u=x+{\bf i}y\in\mathbb{D}$ and $h=a+{\bf i}b\in\mathbb{C}$. Then,
\[
2\mathrm{Im}(\bar{u}h)-s(1-|u|^2)=2(xb-ya)-s(1-x^2-y^2).
\]

If $2\mathrm{Im}(\bar{u}h)-s(1-|u|^2)<0$, we have
\begin{align*}
	F(f_{\xi, u}, (s,h))=& \frac{2|h|-2\mathrm{Im}(\bar{u}h)+s(1-|u|^2)}{1-|u|^2}   \\
=&\frac{2\sqrt{a^2+b^2}}{1-x^2-y^2}+\frac{s(1-x^2-y^2)-2(xb-ya)}{1-x^2-y^2}.
 \end{align*}

For convenience, denote
\[
C=\frac{2}{1-x^2-y^2} \ \ \ \text{and}  \ \ \ r=\sqrt{a^2+b^2}.
\]
After calculation, it can be obtained that
\[
G^{-}=\begin{pmatrix}
		1 & C(\frac{a}{r}+y)& C(\frac{b}{r}-x)\\
		C(\frac{a}{r}+y) & C^2(\frac{a}{r}+y)^2+\frac{CF}{r^3}b^2 & C^2(\frac{a}{r}+y)(\frac{b}{r}-x)-\frac{CF}{r^3}ab   \\
		C(\frac{b}{r}-x) &C^2(\frac{a}{r}+y)(\frac{b}{r}-x)-\frac{CF}{r^3}ab  &C^2(\frac{b}{r}-x)^2+\frac{CF}{r^3}a^2
	\end{pmatrix}.
\]
Let
\[
Q=\begin{pmatrix}
		1 & C(\frac{a}{r}+y)& C(\frac{b}{r}-x)\\
		0 & 1 & 0   \\
		0 & 0  &1
	\end{pmatrix}.
\]
Then,
\[
G^{-}=Q^T\cdot\begin{pmatrix}
		1 & 0& 0\\
		0 & \frac{CF}{r^3}b^2 & -\frac{CF}{r^3}ab   \\
		0 & -\frac{CF}{r^3}ab  &\frac{CF}{r^3}a^2
	\end{pmatrix}\cdot Q.
\]
Notice that
\[
M=\begin{pmatrix}
		\frac{CF}{r^3}b^2 & -\frac{CF}{r^3}ab   \\
		-\frac{CF}{r^3}ab  &\frac{CF}{r^3}a^2
	\end{pmatrix}=
\frac{CF}{r^3}\cdot\begin{pmatrix}
b^2 & -ab   \\
		-ab  & a^2
	\end{pmatrix}
\]
is positive semi-definite and
\[
\left\{\mathbf{u};~\mathbf{u}^TM\mathbf{u}=0 \right\}=\left\{t(0,a,b)^T;~t\in\mathbb{R}   \right\}.
\]
Then, $G^{-}$ is positive semi-definite and
\begin{align*}
\left\{\mathbf{v};~\mathbf{v}^TG\mathbf{v}=0 \right\}=& \left\{Q^{-1}(0,ta,tb)^T;~t\in\mathbb{R} \right\} \\
=&\left\{t\left(\frac{2(xb-ya-r)}{1-r^2},a,b \right)^T;~t\in\mathbb{R} \right\}.
\end{align*}
Since
\[
2(xb-ya)-(1-r^2)\cdot\frac{2(xb-ya-r)}{1-r^2}=2r\ge0,
\]
$G^{-}$ is positive definite on $E^{-}$. Similarly, one can see that the fundamental tensor $G^{+}$ is positive definite on $E^{+}$. So, $F|_{TM\setminus E_0}$ satisfies the condition of strong convexity.

In addition, it is not difficult to see that the fundamental tensor is not independent of the tangent vectors (i.e., the Cartan tensor is non-trivial).
Therefore, $F(f_{\xi, u}, (s,h))$ is a non-Riemannian reversible almost regular Finsler structure.

\textbf{(4)} Finally, let us show that the metric space $(\textrm{Aut}(\mathbb{D}),d_{H^\infty})$ is not a geodetic space.
Given any $f_{\xi, u}\in\textrm{Aut}(\mathbb{D})$. Let
\[
\Omega=\left\{f_{\eta, v}; \left|\mathrm{e}^{\mathbf{i}(\eta-\xi)}\cdot\frac{1-\bar{u}v}{1-u\bar{v}}+1\right|<2\rho(u,v) \right\}.
\]
Notice that $\Omega$ is a non-empty open subset of $\textrm{Aut}(\mathbb{D})$. Choose an element $f_{\eta, v}\in\Omega$. Then, $d_{H^\infty}(f_{\xi, u},f_{\eta, v})=2$.
Given any curve $\Gamma$ connecting $f_{\xi, u}$ and $f_{\eta, v}$ in $\textrm{Aut}(\mathbb D)$. Since $\Omega$ is an open neighborhood of $f_{\eta, v}$, there is an element $f_{\vartheta, w}$ in $\Gamma\bigcap\Omega$. Then, we also have $d_{H^\infty}(f_{\xi, u}, f_{\vartheta, w})=2$ and $d_{H^\infty}(f_{\eta, v},f_{\vartheta, w})>0$. We use $L_{\Gamma}(f_{\xi, u}, f_{\eta, v})$ and $L_{\Gamma}(f_{\vartheta, w}, f_{\eta, v})$  to denote the length of the curve segment from $f_{\xi, u}$ to $f_{\vartheta, w}$ along $\Gamma$ and the length of the curve segment from $f_{\vartheta, w}$ to $f_{\eta, v}$ along $\Gamma$, respectively. Then,
\begin{align*}
	L_{\Gamma}=& L_{\Gamma}(f_{\xi, u}, f_{\eta, v})+L_{\Gamma}(f_{\vartheta, w}, f_{\eta, v}) \\
\geq & d_{H^\infty}(f_{\xi, u}, f_{\vartheta, w})+d_{H^\infty}(f_{\eta, v},f_{\vartheta, w})\\
	= & 2+d_{H^\infty}(f_{\eta, v},f_{\vartheta, w}) \\
>&2=d_{H^\infty}(f_{\xi, u},f_{\eta, v}),
\end{align*}
which implies that the metric space $(\textrm{Aut}(\mathbb{D}),d_{H^\infty})$ is not a geodetic space.
\end{proof}

\begin{remark}
Let us consider the sub-manifold $\mathfrak{D}$ of $(\textrm{Aut}(\mathbb{D}),d_{H^\infty})$ defined by
\[
\mathfrak{D}=\left\{\varphi_{u}(z)=\frac{z-u}{1-\bar{u}z};~ u\in\mathbb{D} \right\},
\]
which could be seemed as the space $\mathbb{D}$ equipped with the metric induced by $d_{H^\infty}$.
Following from the Main Theorem, one can see that, for any $u,v\in\mathbb{D}$,
\begin{enumerate}
  \item if $2-\bar{u}v-\bar{v}u\leq 2|u-v|$,
  \begin{equation*}
    d_{H^\infty}(\varphi_{u},\varphi_{v})=2;
  \end{equation*}
  \item if $2-\bar{u}v-\bar{v}u>2|u-v|$,
  \begin{align*}
      &d_{H^\infty}(\varphi_{u},\varphi_{v}) \\
      =&\left|(1-2\rho^2(u,v)-2\mathbf{i}\rho(u,v)\sqrt{1-\rho^2(u,v)})
      -\left(\textrm{Re}\left(\frac{1-\bar{u}v}{1-u\bar{v}}\right)+\mathbf{i}\left|\textrm{Im}\left(\frac{1-\bar{u}v}{1-u\bar{v}}\right)\right|\right) \right|.
  \end{align*}
  \end{enumerate}
Furthermore, the induced geometric structure $F(\varphi_{u}, h)$ is a non-Riemannian reversible almost regular Finsler structure, whose explicit formula is as follows,
\[
F(\varphi_{u}, h)=\frac{2|h|+2|\mathrm{Im}(\bar{u}h)|}{1-|u|^2}.
\]
$F(\varphi_{u}, h)$ is $C^{\infty}$ except for the directions parallel to $u$. In fact, $F(\varphi_{u}, h)$ is an almost regular Randers metric. Write $u=x+{\bf i}y\in\mathbb{D}$ and $h=a+{\bf i}b\in\mathbb{C}$. Suppose that $\mathrm{Im}(\bar{u}h)>0$. Let
\[
\alpha(\varphi_{u}, h)=\frac{2\sqrt{a^2+b^2}}{1-x^2-y^2} \ \ \ \text{and} \ \ \ \beta(\varphi_{u}, h)=\frac{2(xb-ya)}{1-x^2-y^2}.
\]
Then $F(\varphi_{u}, h)$ is the sum of a Riemannian metric $\alpha(\varphi_{u}, h)$ and a non-trivial $1$-form $\beta(\varphi_{u}, h)$ on $\mathrm{Im}(\bar{u}h)>0$, where
\[
\|\beta\|_{\alpha}=\frac{1-x^2-y^2}{2}\cdot\sqrt{\frac{4(x^2+y^2)}{(1-x^2-y^2)^2}}=\sqrt{x^2+y^2}<1.
\]
Similarly, $F(\varphi_{u}, h)$ is also a Randers metric on $\mathrm{Im}(\bar{u}h)>0$.

\begin{figure}[h]
	\centering
	\includegraphics[width=0.6\textwidth]{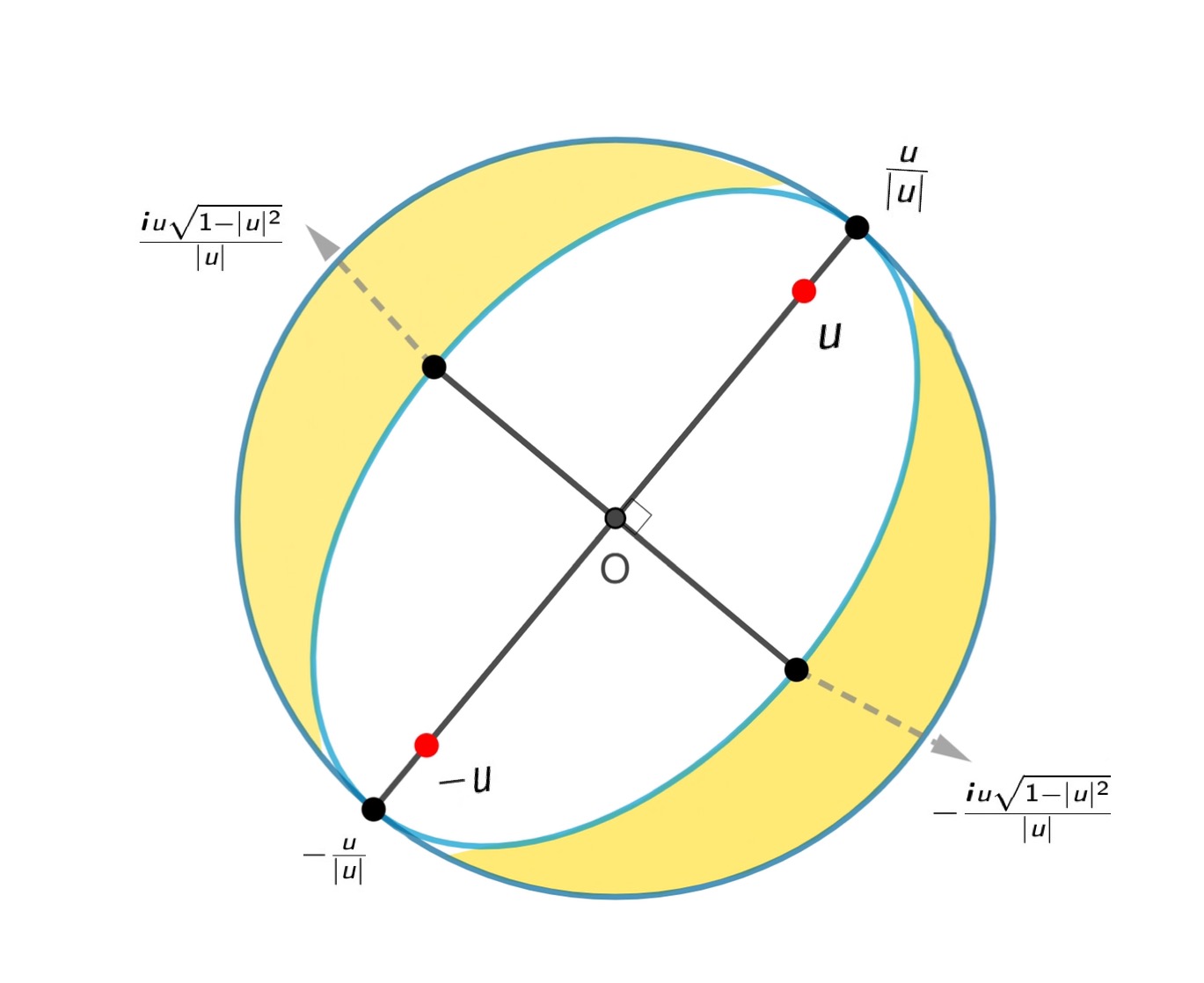}\\
	\caption{}\label{pic2}
\end{figure}

Given any $u\in\mathbb{D}$. As show in Figure \ref{pic2}, the set $\{v\in\mathbb{C};~2-\bar{u}v-u\bar{v}= 2|u-v|\}$ is the ellipse $\Lambda$ with focal points $u$ and $-u$, a major semi-axis length of $1$ and a minor semi-axis length of $\sqrt{1-|u|^2}$. Notice that $d_{H^\infty}(\varphi_{u},\varphi_{v})=2$ if and only if $v$ is not inside of the ellipse $\Lambda$. Then,
similar to the proof for $(\textrm{Aut}(\mathbb{D}),d_{H^\infty})$, one can see that the metric space $(\mathfrak{D}, d_{H^\infty})$ is also not a geodetic space.
\end{remark}

\section*{Declarations}

\noindent \textbf{Ethics approval}

\noindent Not applicable.

\noindent \textbf{Conflict of interest/Competing interests}

\noindent The authors declare that there is no conflict of interest or competing interest.

\noindent \textbf{Funding}

\noindent The third author was partially supported by National Natural Science Foundation of China (Grant No. 12471120).

\noindent \textbf{Authors' contributions}

\noindent All authors contributed equally to this work.

\noindent \textbf{Availability of data and materials}

\noindent Data sharing is not applicable to this article as no data sets were generated or analyzed during the current study.

\end{document}